\documentclass[12pt]{amsart}
\usepackage{amsfonts,amssymb,amscd,amsmath,amsrefs,latexsym,verbatim,calc,
quotmark,enumitem}

\newcommand{\mbf}{\mathbf}
\def\srnk{{\operatorname{s-rank}}}

\def\hom{\operatorname{Hom}}

\def\max{\operatorname{Max}}
\def\dim{\operatorname{dim}}
\def\rank{\operatorname{rank}}
\def\E{{\operatorname{E}}}            
\def\GL{{\operatorname{GL}}}          
\def\og{{\operatorname{O}}}          
\def\EO{{\operatorname{EO}}}          

\newtheorem{theorem}{Theorem}[section]
\newtheorem{lemma}[theorem]{Lemma}
\newtheorem{corollary}[theorem]{Corollary}
\newtheorem{proposition}[theorem]{Proposition}
\theoremstyle{definition}

\newtheorem{definition}[theorem]{Definition}

\newtheorem*{ackn}{Acknowledgements}

\newcommand{\inD}[1][\relax]{\def\argone{#1}\def\temprelax{\relax}
  \ifx\argone\temprelax\right.\else\,\middle|#1\right.{}\fi}

\title{Normality and $K_1$-stability of Roy's elementary orthogonal group }
\author{A. A. Ambily}
\address{\newline Statistics and Mathematics Unit
\newline Indian Statistical Institute
\newline Bangalore 560 059
\newline India.
\newline {\emph{e-mail: } \tt ambily@isibang.ac.in}}

\begin{document}
\maketitle
\begin{abstract}
 In this paper, we prove the normality of the  
 Roy's elementary orthogonal group (Dickson--Siegel--Eichler--Roy or DSER group) 
 over a commutative ring  which was introduced by A. Roy in 
 \cite{MR0231844} under some conditions  on the hyperbolic rank. 
 We also establish a stability theorem for $K_1$ of  Roy's group. 
 We obtain a decomposition theorem for the elementary orthogonal group 
 which is used to deduce the stability theorem. 
\end{abstract}

\section{Introduction}

In 1960's, H. Bass initiated the study of the normal subgroup structure 
of linear groups. He  introduced a new notion of dimension of rings, 
called stable rank, and proved that the principal structure theorems hold 
for groups whose degrees are large with respect to the stable rank. Later, 
J.~S. Wilson, I.~Z. Golubchik and A.~A. Suslin made many other important 
contributions in this direction. In 1977, A.~A. Suslin proved that over 
any commutative ring $A$, the group $\E_n(A)$ is always normal in 
$\GL_n(A)$ when $n \ge 3$.

The normal subgroup structure of symplectic and classical unitary groups over 
rings were studied by V.~I. Kope{\u\i}ko in \cite{MR497932}, G. Taddei in \cite{MR676359}
and by Suslin-Kope{\u\i}ko in \cite{MR0469914}. Similar results were obtained for general
quadratic groups by A. Bak, V. Petrov, and G. Tang in \cite{MR2061845}, for general 
Hermitian groups by G. Tang in \cite{MR1609905} and A. Bak and G. Tang in \cite{MR1765866}, 
and for odd unitary groups by V. Petrov in \cite{MR2033642} and W. Yu in \cite{MR3003312}.

The stability problem for $K_1$ of quadratic forms was studied in 1960's and in early 
1970's by H. Bass, A. Bak, A. Roy, M. Kolster and L.~N. Vaserstein. The stability 
theorems relate unitary groups and their elementary subgroups in different ranges. 
The stability results for quadratic $K_1$ are due to A. Bak, V. Petrov and G. Tang 
(see \cite{MR2061845}), and for Hermitian $K_1$ are due to A. Bak and G. Tang (see \cite{MR1765866}). 
Recently, in \cite{MR3003312}, W. Yu proved the $K_1$-stability for odd unitary groups
which were introduced by V. Petrov. Stronger results for spaces over semilocal rings are
due to A. Roy and M.~Knebusch for quadratic spaces (see \cites{MR0231844,MR0252382}) 
and H. Reiter for Hermitian spaces (see \cite{MR0387276}).

In this paper, we study the Dickson--Siegel--Eichler--Roy (DSER) orthogonal group over
a commutative ring $A$, defined by A. Roy in \cite{MR0231844}. We prove results on 
normality as well as on stability. There are three normality theorems that we will 
prove and the stability results will be obtained under Bass's stable range condition. 
A useful tool in the proof will be a decomposition theorem for the elementary subgroup 
that we will establish in Section~\ref{decomp}. For proving stability, we adapt the method used in \cite{MR2061845,MR1765866}. We also need some commutator relations which we have proved
in \cite{aa}.


\section{Preliminaries}
Let $A$ be a commutative ring with identity in which 2 is invertible. Let $Q$ be a 
quadratic space and $P$ be a finitely generated projective module. Let $M=Q \!\perp\! H(P)$, 
where $H(P)$ denote the hyperbolic space. Let $\og_A(Q \!\perp\! H(P))$ denote the orthogonal 
group of the quadratic space $Q \!\perp\! H(P)$. Here, $Q$ is equipped with a non-singular 
quadratic form and $H(P)$ has the natural hyperbolic form. Let $\EO_A(Q\!\perp\! H(P)$ be 
the subgroup of $\og_A(Q\!\perp\! H(P))$ generated by $E_{\alpha},E_{\beta}^*$ for $\alpha 
\in \hom(Q,P), \beta \in \hom(Q,P^*)$.

We now recall some basic definitions from \cite{aarr}. Assume that $Q$ and $P$ are 
free modules of rank $n$ and $m$ respectively. In \cite{aarr}, by fixing bases for $Q$ and $P$, 
we defined the maps $\alpha_{ij} \in \hom(Q,P)$, $\beta_{ij}^* \in \hom(Q,P^*)$ for 
$\alpha \in \hom(Q,P)$ and $\beta \in \hom(Q,P^*)$. Then we extended these maps to $Q \oplus P \oplus P^*$ 
as follows: 
    \begin{align*}
	\alpha_{ij}(z,x,f) 
	&= \eta_i\circ p_i\circ\alpha\circ \eta_j\circ p_j(z,x,f) 
	 = \left( 0,\langle w_{ij},z \rangle x_i,0 \right), \\
	\beta_{ij}(z,x,f) 
	&= \eta_i\circ p_i\circ\beta\circ \eta_j\circ p_j(z,x,f) 
	 = \left(0,0,\langle v_{ij},z \rangle f_i\right),
    \end{align*}
      where $\{z_i : 1\leq i\leq n \}$ is a basis for $Q$, $\{x_i : 1\leq i\leq m \}$ 
      is a basis for $P$, $\{f_i : 1\leq i\leq m \}$ is a basis for $P^*$ and 
      $w_{ij},\, v_{ij} \in Q$.

In terms of these bases, the elementary orthogonal transformations 
$E_{\alpha_{ij}}$ and $E_{\beta_{ij}}^*$ for $1 \leq i \leq m$, $1 \leq j \leq n$ 
are given as follows:
   \begin{align*} 
      E_{\alpha_{ij}}(z,x,f) 
      &= \left(I-\alpha^*_{ij}+\alpha_{ij}-\frac{1}{2}\alpha_{ij}\alpha^*_{ij}\right)(z,x,f)\\
      &= \left(z-\langle f,x_i \rangle w_{ij},\;x+\langle w_{ij},z\rangle x_i -\langle f,x_i \rangle 	 		    q(w_{ij})x_i,\;f\right);\\
      E_{\beta_{ij}}^*(z,x,f) 
      &= \left(I-\beta_{ij}^*+\beta_{ij}-\frac{1}{2}\beta_{ij}\beta_{ij}^*\right)(z,x,f)\\
      &= \left(z-\langle f_i,x \rangle v_{ij},\;x,\;f+\langle v_{ij},z\rangle f_i -\langle x,f_i \rangle 		   q(v_{ij})f_i\;\right).
   \end{align*}

We denote $Q \!\perp\! H(P)$ by $Q \!\perp\! H(A)^m$, when $\rank(P) = m$. There is a natural 
embedding $\og_A(Q \!\perp\! H(A)^{m-1})\longrightarrow \og_A(Q \!\perp\! H(A)^{m})$ of groups  
which is called the stabilization homomorphism and is given by the following matrix equation.

\vspace{2mm}

Let $\begin{pmatrix}
    \mbf{a}'&\mbf{b}'&\mbf{c}'\\
    \mbf{d}'&\mbf{e}'&\mbf{f}'\\
    \mbf{g}'&\mbf{h}'&\mbf{j}'
     \end{pmatrix}$ be an element of $\og_A(Q\!\perp\! H(A)^{(m-1)})$. Then
\begin{equation}\label{stable}
\begin{pmatrix}
    \mbf{a}'&\mbf{b}'&\mbf{c}'\\
    \mbf{d}'&\mbf{e}'&\mbf{f}'\\
    \mbf{g}'&\mbf{h}'&\mbf{j}'
     \end{pmatrix} \mapsto \begin{pmatrix}
    \mbf{a}'&\vline&\mbf{b}'&0&\vline&\mbf{c}'&0\\\hline
    \mbf{d}'&\vline&\mbf{e}'&0&\vline&\mbf{f}'&0\\
    0&\vline&0&1&\vline&0&0\\\hline
    \mbf{g}'&\vline&\mbf{h}'&0&\vline&\mbf{j}'&0\\
    0&\vline&0&0&\vline&0&1\\
     \end{pmatrix} = \begin{pmatrix}
    \mbf{a}&\mbf{b}&\mbf{c}\\
    \mbf{d}&\mbf{e}&\mbf{f}\\
    \mbf{g}&\mbf{h}&\mbf{j}
     \end{pmatrix}.
\end{equation}
Using this equation, we define the stable orthogonal group and elementary 
orthogonal group as follows: 
\[\og_A =  \displaystyle{\lim_{m\rightarrow \infty}\og_A(Q\!\perp\! H(A)^m)}\mbox{ and }\]
\[\EO_A = \displaystyle{\lim_{m\rightarrow \infty}\EO_A(Q\!\perp\! H(A)^m)}.\]
We now define $$KO_{1,m}(Q \!\perp\! H(A)^{m}) = \og_A(Q \!\perp\! H(A)^m)/\EO_A(Q\!\perp\! H(A)^m),$$ 
which is a coset space. 

We now recall some basic definitions.
\begin{definition}[\cite{MR2235330}*{Chapter~I}]
 Let $A$ be a commutative ring with identity. A vector $(a_1,\ldots,a_n)$ with 
coefficients $a_i \in A$ is called unimodular if there are elements 
$b_1,\ldots, b_n \in A$ such that \[a_1b_1+\ldots+a_nb_n = 1.\]
\end{definition}

\begin{definition}[\cite{MR0284476}]
The ring $A$ is said to satisfy Bass's {\bf stable range condition} $SA_l$ 
in the formulation of L.~N. Vaserstein if, whenever $(a_1,\ldots a_{l+1})$ 
is a unimodular vector, there exist elements $b_1,\ldots, b_l \in A$ such 
that $(a_1+a_{l+1}b_1, \ldots, a_l+a_{l+1}b_l)$ is unimodular. It follows 
easily that $SA_l \Rightarrow SA_k$ for any $k \ge l$.
\end{definition}

\begin{definition}[\cite{MR2235330}*{p.320}]
The {\bf stable rank}, $\srnk$ $A$, of $A$ is defined to be the smallest 
positive integer $l$ such that $A$ satisfies $SA_l$. If no such $l$ exists, 
then the stable rank of $A$ can be taken to be infinite. If $A$ is a local 
ring, $\srnk$ $A$ $= 1$.
\end{definition}

In this paper, we prove the following {\bf normality theorems}.

\vspace{1mm}

\begin{enumerate}[label=(\roman*)]
 \item {\it $ \og_A(Q\!\perp\!H(A)^{m-1})$ normalizes $ \EO_A(Q\!\perp\! H(A)^{m}).$ In particular$,$ $\EO_A$ is a normal subgroup of $\og_A$.}
\item {\it If $m \geq \dim \max(A) + 2,$ then $ \og_A(Q\!\perp\!H(A)^{m})$ normalizes $
\EO_A(Q\!\perp\!H(A)^{m})$.}
\item {\it If $m > l,$ then $\og_A(Q\!\perp\!H(A)^{m})$ normalizes 
$\EO_A(Q\!\perp\!H(A)^{m})$  provided $A$ satisfies the stable range condition $SA_l$.}
\end{enumerate}

Using these normality theorems, we establish the following stability theorem for $K_1$.\\
{\it Suppose $A$ satisfies the stable range condition $SA_l$. Then, for all $m > l$, 
$KO_{1,m}(Q\!\perp\! H(A)^m)$ is a group. Further, the canonical map
 \[
   KO_{1,r}(Q \!\perp\! H(A)^{r}) \longrightarrow KO_{1,m}(Q\!\perp\! H(A)^m) 
 \] is surjective for $l\leq r < m$, and the canonical homomorphism 
 \[
   KO_{1,m}(Q \!\perp\! H(A)^m) \longrightarrow KO_{1,m+1}(Q \!\perp\! H(A)^{m+1})
 \] is an isomorphism.}

A key tool used in the proofs of the above theorems is a decomposition theorem 
for $\EO_A(Q\!\perp\! H(A)^{m})$. The decomposition involves the following subgroups.
\begin{align*}
C_m &= \left\langle [E_{\alpha_{ij}} , E_{\beta_{mk}}^* ] , [E_{\beta_{ij}}^* , E_{\gamma_{mk}}^* ] ,  E_{\beta_{mj}}^* : 1\leq i < m , 1\leq j,k \leq n \right\rangle,\\
D_m &= \left\langle [E_{\alpha_{mj}} , E_{\beta_{ik}}^* ] , [E_{\alpha_{mj}} , E_{\delta_{il}} ] , E_{\alpha_{mj}} :  1\leq i < m , 1\leq j,k \leq n \right\rangle,\\
G_m &= \left \langle E_{\beta_{jk}}^*, [E_{\alpha_{ir}}, E_{\beta_{jk}}^*],
             [E_{\beta_{ir}}^*, E_{\gamma_{jk}}^*] : 1 \leq i,j \leq m, 1 \leq r,k \leq m
	\right \rangle,\\
F_m &= \left \{ \eta\eta_1 : \eta \in \EO_A(Q\!\perp\! H(A)^{m-1}) \mbox{ and } \eta_1 \in C_m \right \}.
\end{align*}
It can be easily observed that $C_m \subseteq G_m$.

\begin{definition}
 Let $\theta \in \EO_A(Q \!\perp\! H(A)^m)$, where $Q$ has rank $n$.  A $FDG$-decomposition of $\theta$ is a product decomposition $\theta = \eta\xi\mu$, where $\eta \in F_m, \xi \in D_m$ and $\mu \in G_m$. 
 An $FDG$-decomposition $\theta = \eta\xi\mu$ will be called {\it reduced}  if the $(n+m-1,n+m)^{th}$ coefficient of $\eta$ is $0$.
\end{definition}

In Section~\ref{decomp}, we prove a decomposition theorem for $\EO_A(Q \!\perp\! H(A)^m)$.

\section{Roy's elementary group is normalized by smaller orthogonal group}
In this section, we prove that $ \og_A(Q\!\perp\!H(A)^{m-1})$ normalizes 
$\EO_A(Q\!\perp\!H(A)^{m})$.

Now, by \cite{aarr}*{Lemma~$3.4$}, each $E_{\alpha}, E_{\beta}^*$ for 
$\alpha \in \hom(Q,P)$ and $\beta \in \hom(Q,P^*)$ can be written as  
a product of $E_{\alpha_{ij}}, E_{\beta_{ij}}^*, 1\leq i\leq m ,1\leq j \leq n$. 
Hence we can consider $\EO_A(Q\!\perp\! H(P))$ as the group generated by 
$E_{\alpha_{ij}}$'s and  $E_{\beta_{ij}}^*$'s for $\alpha \in \hom(Q,P)$ and $\beta \in \hom(Q,P^*)$.

Now, by \cite{aarr}*{Lemma~$4.3$} and the commutator relations which we proved 
in \cite{aa}, we note the following useful interpretation.

\vspace{2mm}
\noindent
{\it The elementary orthogonal group $\EO_A(Q\!\perp\!H(P))$ is generated by 
the elements of the type $E_{\alpha_{ij}}, E_{\beta_{kl}}^* , 
\left[E_{\alpha_{ij}}, E_{\delta_{kl}} \right] , 
\left[E_{\alpha_{ij}}, E_{\beta_{kl}}^* \right] , 
\left[ E_{\gamma_{ij}}^*, E_{\beta_{kl}}^* \right]$ for 
$1\leq i,k \leq m , 1\leq j,l \leq n$ and $i \neq k $.}

\begin{proposition}\label{normal}
$ \og_A(Q \!\perp\! H(A)^{m-1})$ normalizes $ \EO_A(Q \!\perp\! H(A)^{m})$.
\end{proposition}

Towards this, we recall some of the commutator relations
which we proved in \cite{aa}*{Lemma~3.10, 3.14, 3.18, 3.19, 3.20}.

\begin{lemma}\label{lem1}
Let $\alpha,\delta, \xi \in  \hom(Q,P)$ and $\beta,\gamma, \mu \in \hom(Q,P^*)$.
Then, for any given $i,j,k,l$ such that  $1\leq i,j,t \leq m$ and $1 \leq k,l,r,s \leq n$,
we have the following commutator relations.
\addtolength{\jot}{1em}
\begin{enumerate}[label={\em(\roman*)}]
\item  $\left[E_{\beta_{ik}}^* , \left[E_{\alpha_{ir}} , E_{\gamma_{jl}}^* \right]\right]
= E_{\eta_{jk}}^* \left[ E_{\nu_{jk}}^* , E_{\zeta_{ik}}^* \right]$, 
where $\eta_{jk} = - \gamma_{jl} \alpha_{ir} \beta_{ik}^*$,
$\nu_{jk} = -\frac{1}{2} \gamma_{jl} \alpha_{ir} \beta_{ik}^*,$\\
$\zeta_{ik} = - \beta_{ik} $ and $i \neq j$.
\item  $\left[E_{\beta _{ik}}^* , \left[E_{\alpha_{ir}} , E_{\delta_{jl}}\right]\right]
= E_{\lambda_{jk}}\left[E_{\xi_{jk}} , E_{\zeta_{ik}}^*\right]$,
where $\lambda_{jk} = \delta_{jl} \alpha_{ir}^* \beta_{ik}, \xi_{jk} = \frac{1}{2} \delta_{jl} \alpha_{ir}^* \beta_{ik}$,\\
$\zeta_{ik} = \beta_{ik}$ and $i \neq j$.
\item $\left[\left[E_{\beta_{ir}}^* , E_{\gamma_{jl}}^*\right] , \left[E_{\alpha_{js}} , E_{\mu_{tk}}^* \right]\right]
= \left[E_{\zeta_{il}}^* , E_{\nu_{ts}}^* \right]$,
where $\zeta_{il} = -\beta_{ir} {\gamma_{jl}}^*$, $\nu_{ts} = \mu_{tk}{\alpha_{js}}^* $
and for $i,j,t$ distinct.
\item
$\left[\left[ E_{\alpha_{ir}} , E_{\delta_{jl}}\right] , \left[ E_{\xi_{tk}} , E_{\beta_{js}}^* \right]\right]
= \left[ E_{\lambda_{il}} , E_{\eta_{ts}} \right]$, 
where $\lambda_{il} = \alpha_{ir} {\delta_{jl}}^*$, $\eta_{ts} = \xi_{tk}{\beta_{js}}^* $
and for $i,j,t$ distinct.
\item $\left[\left[ E_{\alpha_{ir}} , E_{\beta_{jl}}^* \right] , \left[ E_{\delta_{js}} , E_{\gamma_{tk}}^* \right] \right]
= \left[ E_{\eta_{il}} , E_{\mu_{ts}}^* \right]$,
where $\eta_{il} = -\alpha_{ir} {\beta_{jl}}^*$,
$\mu_{ts} = \gamma_{tk}{\delta_{js}}^* $ and for $i,j,t$ distinct.
\end{enumerate}
In particular,
\begin{enumerate}[label={\em(\roman*)}]
\item
$ E_{\mu_{kj}}^* = \left[E_{\beta_{mj}}^* , \left[ E_{\alpha_{mr}} , E_{\gamma_{kl}}^* \right]\right] {\left[ E_{\nu_{kj}}^* , E_{\zeta_{mj}}^* \right]}^{-1}$,
\vspace{1mm}
\item
$ E_{\lambda_{kj}} = \left[ E_{\beta_{mj}}^* , \left[ E_{\alpha_{mr}} , E_{\delta_{kl}} \right]\right]
{\left[ E_{\xi_{kj}} , E_{\zeta_{mj}}^* \right]}^{-1}$,
\vspace{1mm}
\item$ \left[ E_{\zeta_{il}}^* , E_{\nu_{ks}}^*\right] = \left[\left[ E_{\beta_{ir}}^* , E_{\gamma_{ml}}^* \right] , \left[ E_{\alpha_{ms}} , E_{\mu_{kt}}^* \right]\right]$,
\vspace{1mm}
\item$ \left[ E_{\lambda_{il}} , E_{\eta_{ks}}\right] = \left[ \left[ E_{\alpha_{ir}} , E_{\delta_{ml}}\right] , \left[ E_{\xi_{kt}} , E_{\beta_{js}}^* \right]\right]$,
\vspace{1mm}
\item$ \left[ E_{\eta_{il}} , E_{\mu_{ks}}^* \right] = \left[ \left[ E_{\alpha_{ir}} , E_{\beta_{ml}}^* \right] , \left[ E_{\delta_{ms}} , E_{\gamma_{kt}}^* \right]\right]$.
\end{enumerate}
\end{lemma}
\begin{lemma}
The elementary orthogonal group $\EO_A(Q\!\perp\!H(A)^m)$ is generated by those elementary generators
which have $m$ as one of the subscripts.
\end{lemma}
\begin{proof}
 The proof follows from Lemma~\ref{lem1}. The relations in Lemma~\ref{lem1} show that the group $\EO_A(Q\!\perp\! H(A)^m) $ is generated by the elements of type $ E_{\alpha_{mj}} , E_{\beta_{mk}}^* ,
[E_{\alpha_{ij}} , E_{\beta_{mk}}^*]$ , $[ E_{\alpha_{mj}} , E_{\beta_{ik}}^* ] ,
[E_{\alpha_{mj}} , E_{\delta_{il}} ] , [E_{\beta_{ij}}^* , E_{\gamma_{mk}}^* ],$ when $Q$ and $P$ are free $A$-modules.
\end{proof}
As a consequence of the above lemma, it follows that the groups $D_m$ and
$C_m $ generate the elementary group $\EO_A(Q \!\perp\! H(A)^m)$. We now prove a
normality result for the elementary orthogonal group $\EO_A(Q\!\perp\! H(A)^{m})$.
\begin{proposition}\label{normal}
$ \og_A(Q \!\perp\! H(A)^{m-1})$ normalizes $ \EO_A(Q \!\perp\! H(A)^{m})$.
\end{proposition}
\begin{proof}
To prove this, it is sufficient to prove that $D_m $ and $ C_m $
are normalized by ${\og_A(Q \!\perp\! H(A)^{m-1})}$, and we do this by
direct matrix calculation.

We consider the matrix representation of the elements of $\og_A(Q\!\perp\! H(A)^m)$.

\vspace{2mm}

Let $T = \begin{pmatrix}
    \mbf{a}&\mbf{b}&\mbf{c}\\
    \mbf{d}&\mbf{e}&\mbf{f}\\
    \mbf{g}&\mbf{h}&\mbf{j}
     \end{pmatrix} \in \og_A(Q\!\perp\! H(A)^m)$. Then
\begin{equation}\label{identity}
T^t \Psi T = \Psi, \end{equation} where $\Psi = \varphi \!\perp\! \begin{pmatrix}
                                  0&I\\
                                  I&0 \end{pmatrix}$
                                  is the matrix of the quadratic form on $Q \!\perp\! H(A)^m$. Here $\varphi$ denotes the matrix corresponding to the nondegenerate bilinear form and $\begin{pmatrix}
                                  0&I\\
                                  I&0 \end{pmatrix}$ is the matrix of the bilinear form on the hyperbolic space.
This equation is equivalent to the following set of equations.
\begin{align*}
\mbf{a}^t\varphi \mbf{a} +\mbf{g}^t \mbf{d} + \mbf{d}^t \mbf{g} &= \varphi, &  \mbf{b}^t\varphi \mbf{a} +\mbf{h}^t \mbf{d} + \mbf{e}^t \mbf{g} &= 0, &\mbf{c}^t\varphi \mbf{a} +\mbf{j}^t \mbf{d} + \mbf{f}^t \mbf{g} &= 0,\\
\mbf{a}^t\varphi \mbf{b} +\mbf{g}^t \mbf{e} + \mbf{d}^t \mbf{h} &= 0, & \mbf{b}^t\varphi \mbf{b} +\mbf{h}^t \mbf{e} + \mbf{e}^t \mbf{h} &= 0, & \mbf{c}^t\varphi \mbf{b} +\mbf{j}^t \mbf{e} + \mbf{f}^t \mbf{h} &= I,\\
\mbf{a}^t\varphi \mbf{c} +\mbf{g}^t \mbf{f} + \mbf{d}^t \mbf{j} &= 0, & \mbf{b}^t\varphi \mbf{c} +\mbf{h}^t \mbf{f} + \mbf{e}^t \mbf{j} &= I, & \mbf{c}^t\varphi \mbf{c} +\mbf{j}^t \mbf{f} + \mbf{f}^t \mbf{j} &= 0.
\end{align*}
These equations are equivalent to the equation
\[T^{-1} = \begin{pmatrix}
    \varphi^{-1} \mbf{a}^t\varphi&\varphi^{-1}\mbf{g}^t&\varphi^{-1}\mbf{d}^t\\
    \mbf{c}^t\varphi&\mbf{j}^t&\mbf{f}^t\\
    \mbf{b}^t\varphi&\mbf{h}^t&\mbf{e}^t
     \end{pmatrix}.\]
We now consider the generators for the subgroups $D_m$ and $C_m$ of $\EO_A(Q\!\perp\! H(A)^m)$ and prove that they are normalized by an element in $\og_A(Q\!\perp\!H(A)^{m-1})$. 

Consider $T \in \og_A(Q \!\perp\!H(A)^{m-1})$ as an element in $\og_A(Q\!\perp\!H(A)^m)$ by the stabilization homomorphism. Then we conjugate the elementary generators of $\EO_A(Q\!\perp\!H(A)^m)$ and write the conjugated element as a product of elementary generators.

\vspace{2mm}

Corresponding to the elementary generator $E_{\alpha_{mj}}$, we have

\addtolength{\jot}{1em}
     \begin{align*}
       T^{-1} E_{\alpha_{mj}} T &= \begin{pmatrix}
                    I&0&-\phi^{-1}\mbf{a}^t{\alpha_{mj}}^t \mbf{j}\\
                    \mbf{j}^t {\alpha_{mj}} \mbf{a}& I + \mbf{j}^t {\alpha_{mj}} \mbf{b} & \mbf{j}^t {\alpha_{mj}} \mbf{c} - \mbf{c}^t {\alpha_{mj}}^t \mbf{j}- \frac{1}{2} \mbf{j}^t{\alpha_{mj}} {\alpha_{mj}}^* \mbf{j}\\
                    0&0&I-\mbf{b}^t {\alpha_{mj}}^t \mbf{j}
                    \end{pmatrix}\\
                    \vspace*{2cm}
                  &=[E_{\mbf{j}^t{\alpha_{mj}} \mbf{b} \mbf{c}^t \phi},E_{\frac{\mbf{j}^t{\alpha_{mj}}}{2}}][E_{\mbf{c}^t\phi},E_{\mbf{j}^t{\alpha_{mj}}}][E_{\mbf{b}^t\phi}^* , E_{\mbf{j}^t{\alpha_{mj}}}]E_{\mbf{j}^t{\alpha_{mj}} \mbf{a}}.
                  \end{align*}
       Corresponding to the elementary generator $E_{\beta_{mj}}^*$, we have
  \begin{align*}
  T^{-1} E_{\beta_{mj}}^* T &= \begin{pmatrix}
                    I&-\phi^{-1}\mbf{a}^t\beta_{mj}^t \mbf{e} & 0\\
                    0 & I - \mbf{c}^t \beta_{mj}^t \mbf{e} & 0\\
                    \mbf{e}^t \beta_{mj} \mbf{a}& \mbf{e}^t \beta_{mj} \mbf{b} - \mbf{b}^t \beta_{mj}^t \mbf{e} - \frac{1}{2} \mbf{e}^t \beta_{mj} \beta_{mj}^* \mbf{e} & I + \mbf{e}^t \beta_{mj} \mbf{c}
                    \end{pmatrix}\\
                  &= \left[E_{\mbf{e}^t\beta_{mj} \mbf{c} \mbf{b}^t \phi}^*,E_{\frac{\mbf{e}^t\beta_{mj}}{2}}^*\right]\left[E_{\mbf{b}^t\phi}^*,E_{\mbf{e}^t\beta_{mj}}^* \right] \left[E_{\mbf{c}^t \phi},E_{\mbf{e}^t\beta_{mj}}^*\right] E_{(\mbf{e}^t\beta_{mj}\mbf{a})} .
                    \end{align*}
       Corresponding to the elementary generator $[E_{\alpha_{mj}}, E_{\beta_{kl}}^* ]$, we have
\begin{align*}
       T^{-1} [E_{\alpha_{mj}}, E_{\beta_{kl}}^* ] T
                  &= \begin{pmatrix}
                    I& 0 & \phi^{-1} (\mbf{d}^t \beta_{kl}\alpha_{mj}^* \mbf{j})^t\\
                    - \mbf{j}^t \alpha_{mj}\phi^{-1} \beta_{kl}^t \mbf{d} & I - \mbf{j}^t \alpha_{mj} \beta_{kl}^* \mbf{e} & \mbf{f}^t \beta_{kl}\phi^{-1} \alpha_{mj}^t \mbf{j} - \mbf{j}^t \alpha_{mj} \phi^{-1} \beta_{kl}^t \mbf{f}\\
                    0&0& I + \mbf{e}^t \beta_{kl}\alpha_{mj}^*  \mbf{j}
                    \end{pmatrix}\\
                  &= \left[ E_{(\frac{\mbf{j}^t \alpha_{mj}}{2})}, E_{(\mbf{j}^t \alpha_{mj}\phi^{-1} {\beta_{kl}}^t \mbf{f} \mbf{e}^t \beta_{kl} )}\right] \left[E_{(\mbf{j}^t \alpha_{mj})}, E_{(\mbf{f}^t\beta_{kl})}\right]\\
                 &\hspace{5mm}\left[ E_{(\mbf{j}^t \alpha_{mj})},E_{(\mbf{e}^t\beta_{kl})}^* \right] E_{-(\mbf{j}^t \alpha_{mj}\phi^{-1} \beta_{kl}^t \mbf{d})}.
                  \end{align*}
                  Corresponding to the elementary generator $ [E_{\alpha_{ij}}, E_{\beta_{mk}}^* ]$, we have
    \begin{align*}
                  T^{-1} [E_{\alpha_{ij}}, E_{\beta_{mk}}^* ] T
                 &= \begin{pmatrix}
                    I& -\phi^{-1} (\mbf{e}^t \beta_{mk}\alpha_{ij}^* \mbf{g})^t &0\\
                    0& I - \mbf{j}^t \alpha_{ij} \beta_{mk}^* \mbf{e} &0\\
                    \mbf{e}^t \beta_{mk} \alpha_{ij}^*\mbf{g} & \mbf{e}^t \beta_{mk} \alpha_{ij}^* \mbf{h} - \mbf{h}^t \alpha_{ij} \beta_{mk}^* \mbf{e}&I + \mbf{e}^t \beta_{mk}\alpha_{ij}^*  \mbf{j}
                    \end{pmatrix}\\
                 &= \left[ E_{(\mbf{e}^t \beta_{mk}\phi^{-1} \alpha_{ij}^t \mbf{j} \mbf{h}^t \alpha_{ij})}^*, E_{(\frac{\mbf{e}^t\beta_{mk}}{2})}^*\right] \left[E_{(\mbf{h}^t \alpha_{ij})}^*, E_{(\mbf{e}^t \beta_{mk})}^*\right]\\
                          &\hspace{5mm}\left[E_{(\mbf{j}^t \alpha_{ij})}, E_{(\mbf{e}^t\beta_{mk})}^*\right] E_{(\mbf{e}^t \beta_{mk} \phi^{-1} \alpha_{ij}^t \mbf{g})}^*.
                          \end{align*}
                  Corresponding to the elementary generator $[E_{\alpha_{mk}}, E_{\delta_{jl}} ]$, we have
\begin{align*}
                  T^{-1} [E_{\alpha_{mk}}, E_{\delta_{jl}} ] T
                  & = \begin{pmatrix}
                        I& 0 & \phi^{-1} \mbf{g}^t \delta_{jl} \alpha_{mk}^* \mbf{j}\\
                        -\mbf{j}^t \alpha_{mk} \delta_{jl}^* \mbf{g} & I - \mbf{j}^t \alpha_{mk} \delta_{jl}^* \mbf{h} & \mbf{j}^t(\delta_{jl}\alpha_{mk}^* - \alpha_{mk}\delta_{jl}^*)\mbf{j}\\
                        0 & 0 & I + \mbf{h}^t \delta_{jl} \alpha_{mk}^* \mbf{j}
                        \end{pmatrix}\\
                   & = \left[E_{(\frac{1}{2}\mbf{j}^t \alpha_{mk})} , E_{(\mbf{j}^t\alpha_{mk}\delta_{jl}^* \mbf{h} \mbf{j}^t \delta_{jl})}\right]\left[E_{(\mbf{j}^t \alpha_{mk})}, E_{(\mbf{h}^t \delta_{jl})}^*\right]\left[E_{(\alpha_{mk})}, E_{(\mbf{j}^t \delta_{jl})}\right] E_{(-\mbf{j}^t \alpha_{mk} \delta_{jl}^* \mbf{g})}.
                  \end{align*}
                 Corresponding to the elementary generator $[E_{\beta_{mk}}^*, E_{\gamma_{jl}}^* ]$, we have
    \begin{align*}
          T^{-1} [E_{\beta_{mk}}^*, E_{\gamma_{jl}}^* ] T
                  & = \begin{pmatrix}
                        I& \phi^{-1} \mbf{d}^t \gamma_{jl} \phi^{-1} \beta_{mk}^t \mbf{e} & 0\\
                        0 & I + \mbf{f}^t \gamma_{jl} \phi^{-1} {\beta_{mk}}^t \mbf{e} & 0 \\
                        -\mbf{e}^t \beta_{mk} \phi^{-1} \gamma_{jl}^t \mbf{d} & \mbf{e}^t (\gamma_{jl}\beta_{mk}^* - \beta_{mk}\gamma_{jl}^*) \mbf{e}  & I - \mbf{e}^t \beta_{mj} \phi^{-1} \gamma_{jl}^t \mbf{f}
                        \end{pmatrix}\\
                        & = [E_{(\frac{\mbf{e}^t \beta_{mk}}{2})}^* , E_{(\mbf{e}^t\beta_{mk}\gamma_{jl}^* \mbf{e} \mbf{f}^t \gamma_{jl})}^*][E_{(\mbf{e}^t \beta_{mk})}^*,  E_{(\mbf{e}^t \gamma_{jl})}^*]\\ 
                    &\hspace{5mm}[E_{(\mbf{e}^t\beta_{mk})}^*, E_{(\mbf{f}^t \gamma_{jl})}] E_{(-\mbf{e}^t \beta_{mk} \gamma_{jl}^* \mbf{d})}.
      \end{align*}

      Now it follows from the above equations that $C_m$ and $D_m$ are normalized by $\og_A(Q \!\perp\! H(A)^{m-1})$. Hence the proposition follows.
\end{proof}
\begin{corollary}
  $\EO_A $ is a normal subgroup of $ \og_A$.
\end{corollary}

\section{Normality of Roy's elementary group under condition on hyperbolic rank}
In this section, we prove that $\EO_A(Q\!\perp\! H(A)^m)$ is normal in $\og_A(Q\!\perp\! H(A)^m)$ under a condition on the hyperbolic rank. First, we prove the normality when the hyperbolic rank is at least $d+2$, where $d= \dim \max(A)$ .
\begin{theorem}
 $\EO_A(Q\!\perp\! H(A)^m)$ is normal in $\og_A(Q\!\perp\! H(A)^m)$ when $m \geq d+2$, where $d= \dim \max(A)$ .
\end{theorem}
\begin{proof}
 By \cite{MR0231844}*{Theorem~7.1}, it follows that $\EO_A(Q \!\perp\! H(A)^m)$
 acts transitively on hyperbolic pairs. In the case of semi-local rings, by \cite{MR0231844}*{Theorem~$8.1'$}, the same holds for $m \geq 1$.

 For, if $\alpha \in \og_A(Q\!\perp\! H(A)^m)$ and $(e_1,f_1)$ is a hyperbolic pair, then, by \cite{MR0231844}*{Corollary~6.4}, $(\alpha e_1, \alpha f_1)$ and $(e_1,f_1)$ are in the same orbit of $\EO_A(Q\!\perp\! H(A)^m)$. Let $e$ be a map which takes one orbit to the other. Therefore $e\alpha$ fixes $(e_1,f_1)$ and hence $e\alpha \in \og_A(Q\!\perp\! H(A)^{m-1})$, whence so does ${(e \alpha)}^{-1}$. Now, by Proposition~\ref{normal}, it follows that $(e\alpha)^{-1}$ normalizes the elementary orthogonal group $\EO_A(Q \!\perp\! H(A)^m)$. This implies that $\alpha^{-1}$ normalizes $\EO_A(Q \!\perp\! H(A)^m).$
\end{proof}

\section{A decomposition theorem}\label{decomp}

In this section, we prove a decomposition of Roy's elementary orthogonal group
under the stable range condition. We start with the following lemma.
\begin{lemma}\label{slem1}
 Let $m \ge l+1$. Then, for any $\sigma \in \og_A(Q \!\perp\! H(A)^m)$, there is an element $\varrho \in G_m$ such that $\sigma\varrho$ has $1$ in its $(n+m,n+m)^{th}$ position.
\end{lemma}
\begin{proof}
 Let $\sigma$ be the $3 \times 3$ block matrix corresponding to the orthogonal transformation
$\sigma \in \og_A(Q \!\perp\! H(A)^m)$ given by
						$$\sigma = \begin{pmatrix}
							   \sigma_{11}&\sigma_{12}&\sigma_{13}\\
							    \sigma_{21}&\sigma_{22}&\sigma_{23}\\
							     \sigma_{31}&\sigma_{32}&\sigma_{33}
							    \end{pmatrix},$$
               where $\sigma_{11}$ is an $n\times n$ matrix, $\sigma_{12}, \sigma_{13}$ are $n\times m$ matrices, $\sigma_{21},\sigma_{31}$ are $m\times n$ matrices and $\sigma_{22},\sigma_{23},\sigma_{32},\sigma_{33}$ are $m \times m$ matrices. Since $\sigma^{-1} \in \og_A(Q \!\perp\! H(A)^m)$, it also has a similar matrix description. Now $(\sigma_{21}, \sigma_{22}, \sigma_{23})$ is a unimodular vector in $M_n(A)\times (M_m(A))^2$. Let $(u,v,w)$ be the bottom row of $(\sigma_{21}, \sigma_{22}, \sigma_{23})$. Since $\sigma^{-1} \in \og_A(Q \!\perp\! H(A)^m)$, it also has a similar matrix description. Now $(\sigma_{21}, \sigma_{22}, \sigma_{23})$ is a unimodular vector in $M_n(A)\times (M_m(A))^2$. Let $(u,v,w)$ be the bottom row of $(\sigma_{21}, \sigma_{22}, \sigma_{23})$.  It is unimodular in $A^n\times A^{2m}$.
Also, $v$ is unimodular in $A^m \cong P$. Then, by \cite{MR0231844}*{Remark~5.6},
there exists an orthogonal transformation $\mu_1 = E_{\beta}^* \in G_m$
which maps $(u,v,w)$ into a unimodular vector  $(0,v,w') \in H(P)$.

Since $A$ satisfies the stable range condition $SA_l$ and $m \ge l+1$, there exists a matrix $\gamma \in M_m(A)$ such that $v'+w'\gamma$ is unimodular in $A^{m}$. Now set
\[
  \mu_2 = \begin{pmatrix}
           I&0&0\\
        0&I&0\\
        0&\gamma&I
          \end{pmatrix} \in G_m,
\]where $I$ denotes the identity matrix and $0$ denotes the zero matrix of the corresponding block size.

  Since $A$ satisfies stable range condition $SA_l$ and $m \ge l+1$, there is a product $\epsilon$ of elementary matrices such that $(v'+w'\gamma)\epsilon = (0,\ldots,0,1)$.

Set \[
  \mu_3 = \begin{pmatrix}
           I&0&0\\
        0&\varepsilon&0\\
        0&0&\varepsilon^{t^{-1}}
          \end{pmatrix} \in G_m.
\]
Then $\sigma \mu_1\mu_2\mu_3$ has $(n+m)^{th}$ row $(0,0,\ldots,1,w\varepsilon^{t^{-1}})$.
This completes the proof of the lemma.
 \end{proof}

Now we can prove the following decomposition theorem.

\begin{theorem}[{\bf Decomposition Theorem}]
 Let $m \ge l+2$. Then every element of $\EO_A(Q \!\perp\! H(A)^m)$ has a reduced $FDG$-decomposition.
\end{theorem}
\begin{proof}
 We first show that if $\theta$ has an $FDG$-decomposition, then it has a reduced one.

 Let $\eta\xi\mu$ be an $FDG$-decomposition of $\theta$. Write
 \[\eta = \begin{pmatrix}
           \eta_{11}&\eta_{12}&0&\eta_{14}&0\\
           \eta_{21}&\eta_{22}&0&\eta_{24}&0\\
           0&0&1&0&0\\
           \eta_{41}&\eta_{42}&0&\eta_{44}&0\\
           0&0&0&0&1
          \end{pmatrix},
\] where the block matrix $\eta_{11}$ is of size $n\times n$, the block matrices $\eta_{12}, \eta_{14}$ are of size $n\times {(m-1)}$, $\eta_{21}, \eta_{41}$ are of size ${(m-1)}\times n$ and $\eta_{22},\eta_{24},\eta_{42},\eta_{44}$ are of size ${(m-1)}\times {(m-1)}$.
and set
\[\eta_1 =
   \begin{pmatrix}
    \eta_{11}&\eta_{12}&\eta_{14}\\
           \eta_{21}&\eta_{22}&\eta_{24}\\
           \eta_{41}&\eta_{42}&\eta_{44}
     \end{pmatrix}.
\]
By definition, $\eta_1 \in \EO_A(Q\!\perp\! H(A)^{m-1})$.
Since $m \ge l+2$, it follows from Lemma~\ref{slem1} that there is an element $\mu_1 \in G_{m-1}$ such that the $(n+m-1,n+m-1)^{th}$ coefficient of $\eta_1\mu_1$ is $1$. Identifying $\mu_1$ with its image under the stabilization map $$\EO_A(Q\!\perp\! H(A)^{m-1}) \longrightarrow \EO_A(Q\!\perp\! H(A)^{m-1}\!\perp\! H(A)),$$ we have $\mu_1 \in G_m$ and the $(n+m-1,n+m-1)^{th}$ coefficient of $\eta\mu_1$ is $1$. Also, $\mu_1 \in F_m \cap G_m$ and $\mu_1$ normalizes $D_m$. Thus $(\eta \mu_1)(\mu_1^{-1}\xi\mu_1)(\mu_1^{-1}\mu)$ is an $FDG$-decomposition of $\theta$ such that the $(n+m-1,n+m-1)^{th}$ coefficient of $\eta\mu_1$ is $1$.

Choose an element $\alpha \in \hom_A(Q,P)$ such that the $(n+m-1,n+m)^{th}$ coefficient of $\eta\mu_1 [E_{\alpha_{m-1,j}},E_{\beta_{mk}}^*]$ is $0$. Choose $\delta \in \hom_A(Q,P)$ such that the $(n+m,n+m-1)^{th}$ coefficient of $\mu_1^{-1}\xi \mu_1 [E_{\delta_{mj}},E_{\beta_{m-1,k}}^*]$ is $0$.
Let $\mu_2 = [E_{\alpha_{m-1,j}},E_{\beta_{mk}}^*]$ and  $\mu_3 = [E_{\delta_{mj}},E_{\beta_{m-1,k}}^*]$.
Then $$\mu_2^{-1}(\mu_1^{-1} \xi \mu_1\mu_3)\mu_2 = \eta_2 \xi_1$$
for some $\eta_2 \in \EO_A(Q\!\perp\! H(A)^{m-1}) \subseteq F_m$ and some $\xi \in D_m$.
Thus $$\theta = \eta\xi\mu = (\eta\mu_1\mu_2)(\mu_2^{-1}(\mu_1^{-1}\xi\mu_1\mu_3)\mu_2)
(\mu_2^{-1}\mu_3^{-1}\mu_1^{-1}\mu)$$ which is a reduced $FDG$-decomposition of
$\theta$.

We now prove that every element of $\EO_A(Q\!\perp\! H(A)^m)$ does have an $FDG$-decomposition. In order to do this, we consider the generators $\varepsilon$ of $\EO_A(Q \!\perp\! H(A)^m)$ and  show that $\varepsilon F_mD_mG_m \subseteq F_mD_mG_m$. It would follow from this that $\EO_A(Q \!\perp\! H(A)^m) = F_mD_mG_m$. Thus, by the first part of the proof, it is enough to prove that $\varepsilon \eta\xi\mu \in F_mD_mG_m$ for each reduced $FDG$-decomposition $\eta\xi\mu$. The rest of the proof shows this.

The commutator relations in the Lemma~\ref{lem1} show that $F_m$ and the matrices $[E_{\delta_{mj}},E_{\beta_{m-1,k}}^*]$ generate $\EO_A(Q\!\perp\! H(A)^m)$, where $\rank(Q)= n$. 
Evidently, \[F_m(F_mD_mG_m) \subseteq F_mD_mG_m.\]

We now consider an element with a reduced $FDG$-decomposition
$\eta\xi\mu$. Since the $(n+m-1,n+m)^{th}$ coefficient of $\eta$ is
$0$, $\eta$ can be expressed as a product $\eta = \eta_3\eta_4$,
where $\eta_3 \in C_m$ such that the $(n+m-1,n+m)^{th}$ coefficient of
$\eta_3$ is $0$ and $\eta_4 \in \EO_A(Q\!\perp\! H(A)^m)$. By a
straightforward computation, one can show that
$$[E_{\delta_{mj}},E_{\beta_{m-1,k}}^*]\eta_3
[E_{\delta_{mj}},E_{\beta_{m-1,k}}^*]^{-1} \in F_m.$$ Clearly,
$\EO_A(Q\!\perp\! H(A)^m)$ normalizes $D_m$. Thus
\[
[E_{\delta_{mj}},E_{\beta_{m-1,k}}^*]\eta\xi\mu
= ([E_{\delta_{mj}},E_{\beta_{m-1,k}}^*]\eta_3
[E_{\delta_{mj}},E_{\beta_{m-1,k}}^*]^{-1}\eta_4)
(\eta_4^{-1}[E_{\delta_{mj}},E_{\beta_{m-1,k}}^*]\eta_4 \xi)\mu\]
which is an $FDG$-decomposition.
\end{proof}

\section{Normality under stable range}

In this section, we prove the normality under the assumption that $A$ satisfies the stable range
condition $SA_l$.
\begin{theorem}\label{normal2}
 Let $A$ be a commutative ring in which $2$ is invertible. Suppose $A$ satisfies the stable range condition $SA_l$. Then, for all $m > l$, $\EO_A(Q\!\perp\! H(A)^m)$ is normal in $\og_A(Q\!\perp\! H(A)^m)$.
\end{theorem}
\begin{proof}
  Let $\eta \in \EO_A(Q \!\perp\! H(A)^m)$, where $\rank(Q) = n$. By Lemma~\ref{slem1}, there is an element $\varrho_1$ in $G_m \subseteq \EO_A(Q \!\perp\! H(A)^m)$ such that the $(n+m,n+m)^{th}$ coefficient of $\eta\varrho_1$ is $1$. Then there is a matrix $\varrho_2 = \prod_{i=1}^{m-1}[E_{\alpha_{mj}}, E_{\beta_{ik}}^*]$ such that $\eta\varrho_1\varrho_2$ has $0$ in the first $n+m-1$ entries of its $(n+m)^{th}$ row and $1$ in the $(n+m)^{th}$ entry of this row. It follows that there is  a matrix
 $\varrho_3 = \prod_{i=1}^m [E_{\beta_{ir}}^*, E_{\gamma_{mk}}^*] \prod_{i=1}^{m-1}[E_{\alpha_{ir}}, E_{\beta_{mk}}^*]E_{\gamma_{mj}}^*$ such that $\varrho_3\eta\varrho_1\varrho_2$ has the same $m^{th}$ row as $\eta\varrho_1\varrho_2$ and the same $m^{th}$ column as the $(n+2m)\times(n+2m)$ identity matrix. For any matrix
 \[
\sigma= \begin{pmatrix}
 \sigma_{11}&\sigma_{12}&\sigma_{13}\\
                 \sigma_{21}&\sigma_{22}&\sigma_{23}\\
                 \sigma_{31}&\sigma_{32}&\sigma_{33}
\end{pmatrix} \in \og_A(Q\!\perp\! H(A)^m),
 \]
 it follows from the Equation~\eqref{identity} that
we get the $(n+2m,n+2m)^{th}$ coefficient of $\varrho_3\eta\varrho_1\varrho_2$ is $1$. Then there is a matrix 
\[
\varrho_4 =  \prod_{i=1}^{m-1}[E_{\alpha_{mk}}, E_{\beta_{ir}}^*]\prod_{i=1}^m [E_{\beta_{ir}}^*, E_{\gamma_{mk}}^*] \prod_{i=1}^{m}[E_{\alpha_{ir}}, E_{\delta_{mk}}] E_{\zeta_{mj}}
\] such that $\varrho_4\varrho_3\eta\varrho_1\varrho_2$ has the same $(n+m)^{th}$ row and $(n+m)^{th}$ column as $\varrho_3\eta\varrho_1\varrho_2$ and the same $(n+2m)^{th}$ column as the $(n+2m,n+2m)$ identity matrix. Now, it follows that $\varrho_4\varrho_3\eta\varrho_1\varrho_2$ has the same $(n+2m)^{th}$ row as the $(n+2m,n+2m)$ identity matrix. Thus, by the stabilization homomorphism, we have $\varrho_4\varrho_3\eta\varrho_1\varrho_2 \in O_A(Q\!\perp\! H(A)^{m-1})$, where $\rank(Q) = n$.
Let $\rho = \varrho_4\varrho_3\eta\varrho_1\varrho_2$. By Proposition~\ref{normal}, it follows that $\rho$ normalizes $\EO_A(Q\!\perp\! H(A)^m)$, where $\rank(Q)=n$. 
Since $\eta = \varrho_3^{-1}\varrho_4^{-1}\rho\varrho_2^{-1}\varrho_1^{-1}$, it follows that $\eta$ normalizes $\EO_A(Q\!\perp\! H(A)^m)$. Thus $\EO_A(Q\!\perp\! H(A)^m)$ is normal in $\og_A(Q\!\perp\! H(A)^m)$.
\end{proof}

\section{Stability of $K_1$}

In this section, we prove the following stability theorem
using the normality theorem of the previous section and the
decomposition theorem.

\begin{theorem}
 Let $A$ be a commutative ring of stable rank $l$ in which $2$ is invertible.
 Then, for all $m > l$, $KO_{1,m}(Q\!\perp\! H(A)^m)$ is a group. Further, the canonical map
 \[KO_{1,r}(Q \!\perp\! H(A)^{r}) \longrightarrow KO_{1,m}(Q\!\perp\! H(A)^m) \] is surjective for $l\leq r < m$, and the canonical homomorphism
 \[KO_{1,m}(Q \!\perp\! H(A)^m) \longrightarrow KO_{1,m+1}(Q \!\perp\! H(A)^{m+1})\] is an
 isomorphism.
\end{theorem}
\begin{proof}
By Theorem~\ref{normal2}, we get $KO_{1,m}(Q\!\perp\! H(A)^m)$ is a group and the map
$$KO_{1,m-1}(Q\!\perp\! H(A)^{m-1}) \longrightarrow KO_{1,m}(Q\!\perp\! H(A)^m)$$ is surjective.
By induction on $m-l$, we obtain that the map $$KO_{1,r}(Q\!\perp\! H(A)^r) \longrightarrow KO_{1,m}(Q\!\perp\! H(P))$$ is surjective for $l\leq r < m$.

To prove the final assertion, let $\sigma \in \og_A(Q\!\perp\! H(A)^m) \cap \EO_A(Q\!\perp\! H(A)^m\!\perp\! H(A))$. Let $\eta\xi\mu$ be an $F_{(m+1)}D_{(m+1)}G_{(m+1)}$-decomposition of $\sigma$. Since the $(n+m+1)^{th}$ row of $\eta$ coincides with that of the $(n+2(m+1))\times (n+2(m+1))$ identity matrix, it follows that the $(n+m+1)^{th}$ row of $\eta\xi\mu$ coincides with the $(n+m+1)^{th}$ row of $\xi\mu$. Thus the $(n+m+1)^{th}$ row of $\xi\mu$ coincides with that of the $(n+2(m+1))\times (n+2(m+1))$ identity matrix. We can write the matrix $\mu$ as
\[
\mu= \begin{pmatrix}
I&\gamma&0\\
0&\varepsilon&0\\
\vartheta&\psi&\varepsilon^{t^{-1}}
\end{pmatrix},
\]where $I$ is an $n\times n$ identity matrix, $\gamma$ is an $n\times m$ matrix, $\varepsilon$ is an $m\times m$ invertible matrix, $\vartheta$ and $\psi$ are matrices of size $m\times n$ and $m \times m$ respectively.

If $(u,v,w)$ denotes the $(n+m+1)^{th}$ row of $\xi$, then the $(n+m+1)^{th}$ row of $\xi\mu$ is
\[
\begin{pmatrix}
 u,&v,&w
\end{pmatrix}\begin{pmatrix}
          I&\gamma&0\\
          0&\varepsilon&0\\
          \vartheta&\psi&\varepsilon^{t^{-1}}
          \end{pmatrix} = \begin{pmatrix}
 u+w\vartheta,&u\gamma+v\varepsilon+w\psi,&w\varepsilon^{t^{-1}}
\end{pmatrix}
  \]
  The $(n+m+1)^{th}$ row of $\xi\mu$ coincides with that of the $n+2(m+1)\times n+2(m+1)$ identity matrix. Hence $w(\varepsilon^{t})^{-1} = 0$. Since $(\varepsilon^{t})^{-1}$ is invertible, we get $w=0$. This implies that $u=0$. Thus $\xi \in G_{m+1}$.

  Now write $\eta = \eta_1\mu_1$, where $\eta_1 \in \EO_A(Q\!\perp\! H(A)^m)$ and $\mu_1 \in C_{m+1}\subseteq G_{m+1}$.

  Then $\sigma = \eta_1\mu_1\xi\mu$ and $\mu_1\xi\mu \in G_{m+1} \cap \og_A(Q\!\perp\! H(A)^m)$. It suffices to show that $\mu_1\xi\mu \in \EO_A(Q\!\perp\! H(A)^m)$. In fact, we show that $\mu_1\xi\mu \in G_m$.

  Write \[
        \mu_1\xi\mu = \begin{pmatrix}
                       I&\gamma&0\\
                0&\varepsilon&0\\
                \vartheta&\delta&\varepsilon^{t^{-1}}
                      \end{pmatrix}.
     \]
Since $\mu_1\xi\mu \in \og_A(Q\!\perp\! H(A)^m)$, it follows that $\gamma, \delta$ have their last column $0$ and $\vartheta,\delta$ have their last row $0$. Also, it follows that $\varepsilon \in \GL_m(A)$. From the definition of $G_{m+1}$, we see that $\varepsilon$ is an $(m+1)\times(m+1)$ matrix of the form
\[
    \varepsilon = \begin{pmatrix}
                     \varepsilon'&0\\
                     0&1
                    \end{pmatrix} \in \E_{m+1}(A)
 \]

Thus $\varepsilon' \in \E_{m+1}(A)\cap \GL_m(A)$. Since $A$ satisfies the stable range condition, by the stability for $K_1$ of the general linear group, we have $\varepsilon' \in \E_m(A)$.

Thus $\mu_1\xi\mu$ lies in $G_m$. Hence the canonical homomorphism
 \[KO_{1,m}(Q\!\perp\! H(A)^m) \longrightarrow KO_{1,m+1}(Q \!\perp\! H(A)^m)\] is an isomorphism.
\end{proof}

\begin{ackn}
The author would like to acknowledge her deep gratitude to Prof.\,B. Sury for his valuable suggestions. The author would like to thank Prof.\,Guoping Tang for helpful discussions. The author is also indebted to 
Prof.\,Ravi A. Rao for his support and encouragement. 
\end{ackn}

\end{document}